\documentclass{elsarticle}

\usepackage{graphicx}
\usepackage{tikz}
\usetikzlibrary{cd}

\usepackage{todonotes}
\usepackage{amsmath}
\usepackage{amsthm}
\usepackage{amssymb}

\usepackage[hyphens]{url}
\usepackage{hyperref}

\theoremstyle{plain}
\newtheorem{proposition}{Proposition}

\let\today\relax
\makeatletter
\def\ps@pprintTitle{%
    \let\@oddhead\@empty
    \let\@evenhead\@empty
    \def\@oddfoot{\footnotesize\itshape
         {Submitted preprint} \hfill\today}%
    \let\@evenfoot\@oddfoot
    }
\makeatother

\begin{document}

\title{Landscapes of data sets and functoriality  of persistent homology} 


\author[1]{Wojciech Chach\'olski}             

\ead{wojtek@kth.se}

%
%
\address[1]{Mathematics Department,
         KTH,
         Lindstedtsvägen 25,
         Stockholm, 11428,
         Sweden}


\author[2]{Alessandro De Gregorio}
\ead{alessandro.degregorio@polito.it}
\address[2]{Mathematics Department,
         Politecnico di Torino,
         Corso Duca degli Abruzzi 24,
         Torino, 10129, 
         Italy}
\author[3]{Nicola Quercioli \corref{cor1}}
\ead{nicola.quercioli2@unibo.it}
\address[3]{Mathematics Department,
         University of Bologna,
          Piazza di Porta S. Donato 5,
         Bologna, 40126, 
         Italy}

\author[1]{Francesca Tombari}
\ead{tombari@kth.se}


\cortext[cor1]{Corresponding author.}


\begin{abstract}
The aim of this article is to describe a new perspective on functoriality of  persistent homology and explain its intrinsic symmetry that is often overlooked. 
A data set  for us is a finite collection of functions,  called measurements, with a finite domain.
Such  a data set might contain internal symmetries which are effectively  captured by the action of a set of the domain endomorphisms. 
Different choices of the set of endomorphisms encode different symmetries of the data set. 
We describe various category structures on such enriched data sets and prove some of their properties such as decompositions and morphism formations.   We also describe a data structure, based on coloured directed graphs, which is convenient to encode the mentioned enrichment. 
We show that persistent homology preserves only some aspects of these landscapes of enriched data sets however not all. In other words persistent homology is not a functor on the entire  category  of enriched data sets.  Nevertheless we show that 
 persistent homology is functorial   locally. We use the concept of equivariant operators to capture some of  the information missed by persistent homology.
\end{abstract}


\maketitle


\section{Introduction}
In this article we give an answer to the question: what is persistent homology a functor of?

We will consider data sets given by finite sets of functions on a finite set $X$  with real values.
There are several important consequences of data sets having this form. For example,  they endow
$X$ with a pseudometric,  enabling  us to extract non-trivial homological information in form of 
persistent homology, one of the key invariants studied in Topological Data Analysis.
A single measurement does not contain any higher non-trivial homological information.
Sets of measurements however do.  Thus it is essential that measurements, on a given set $X$, are grouped together to form various data sets.
 In this case persistent homology becomes a non-expansive (1-Lipschitz) function $PH_d^{\Phi}\colon \Phi\to \text{Tame}( [0,\infty)\times  {\mathbf R},\text{Vect})$,
 assigning to each measurement in the data set $\Phi$ a tame vector space parametrized by $[0,\infty)\times  {\mathbf R}$.
 It is important to notice that the choice of a set of measurements on $X$ affects the pseudometric defined on it. One can use this fact to change the metric on $X$ in order to extract more meaningful information from persistent homology. For example consider  $X$  to be a finite sample of points on a circle. If  $\Phi$ consists of  only one function given by the  $x-$coordinate, then 
 the persistent homology of this measurement is trivial in degrees greater than $0$. 
 If we enlarge the data set by adding to the $x-$coordinate  the function  given by precomposing $x$ with rotation  by $90$ degrees,
 then the persistent homology of the function $x$ with respect to this bigger data set gains a non-trivial homology  in degree $1$.
 This illustrates how   our  knowledge of an object is affected by the number and the type of measurements done on it. 
 Furthermore in this example we gain additional  information by 
 enlarging the set of measurements through the  action of some of  the endomorphisms of $X$  on the existing measurements. We can then take advantage of these actions to inject geometrical features of our choice on a given data set.
 For exhibiting and extracting interesting homological features of data sets,  such actions are therefore important. 
  
 A data set $\Phi$ is naturally equipped with an action of the monoid of its operations  $\text{End}_{\Phi}(X)$, which are endomorphisms of $X$ preserving $\Phi$. This action gives the set $\Phi$ a structure of Grothendieck graph. Persistent homology turns out to be a
  functor indexed by  this  graph, rather than simply a function. 
  Thus not only   persistent homology can be assigned  to individual measurements  in a data set,
  but 
   operations can be used to compare   persistent homologies of different measurements. 
  That is what we call local functorial properties of  persistent homology. 
  
  Persistent homology also has certain  global functorial properties. 
  There are various ways of representing data in the form of sets of measurements, we might choose different units or  different parametrizations of a domain of measurements, or we might need to focus  only on certain operations such as rotations. 
  Furthermore, the same measurements  might be part of different data sets.  These are  some of the reasons why it is essential to 
  be able to compare data sets  equipped with  different structures. For that purpose we introduce the notion of  incarnations of data sets to
  encode different actions, and SEOs to compare  incarnations. 
An incarnation of   a data set $\Phi$ is  an action of a subset $M\subset \text{End}_{\Phi}(X)$.  
  A SEO (set equivariant operator) between  two incarnations $(\Phi,M)$ and $(\Psi,N)$ is a pair consisting of 
  a map $T:M\to N$ and an equivariant (with respect to $T$) function  $\alpha: \Phi \to \Psi$.
The use of this kind of operators for the comparison of incarnations of data sets has been inspired by~\cite{DBLP:journals/corr/abs-1812-11832,zdfhgsfgjhefjeyghj} , where GENEOs (group equivariant non-expansive operators)  are introduced and used for applications to neural networks. 
  If a SEO is geometric, then there is a comparison map between persistent homologies of the incarnations connected by the SEO. However if a SEO is not geometric, such as the change of units SEO, there is no direct comparison of persistent homologies of the involved incarnations.
 Such SEOs   therefore exhibit  diverse homological features of data sets enhancing the analysis. 
 This suggests  complementarity of these operators and persistent homology for a geometric analysis of a data set.
Consider the change of unit as an example. In general it is the SEO obtained by composing measurements in a data set by a given real valued function defined on the real numbers. Multiplication by $-1$ is an example of such a SEO. It has the effect of turning the sub-level sets persistent homology of a measurement into its super-level sets persistent homology, leading in general to a completely different information about the data set. The outcome consists of two different points of view on the same object, that are not functorially comparable, but together may enhance the accuracy of the analysis of the object of interest.

\section{Data sets}\label{asfdsfghfw}
For us a data set, which we regard as a point in the data landscape, 
is   given by a finite  set of real valued functions on some finite set $X$ also called measurements:
 \[\Phi=\{\phi_i\colon X\to\mathbf{R}\ |\ i=1,\cdots,m\}.\]
We define $\text{dom}(\Phi)$, the \textbf{domain} of  dataset $\Phi$, to be  the set $X$  which is the domain of the functions in $\Phi$. 
 The most fundamental  aspect  of  a data set  $\Phi$  is that it is a set. 
All such  data sets with different domains    form a category with  functions as morphisms. This is the most primitive landscape of  data sets.  The nature of our data sets however can be used to  impose more intricate structures and more meaningful landscapes.
This is reminiscent of the case  of groups. The most fundamental  aspect  of  a group is that it is a set. However the category whose morphisms are group homomorphism is a much more meaningful landscape in which to study relationships between groups.
 To understand relationships between topological groups, the category 
with continuous group homomorphisms  provides an even  
more meaningful landscape.

In this most primitive landscape however   we can already perform products and coproducts.
Let $\phi\colon X\to \mathbf{R}$ and $\psi\colon Y\to \mathbf{R}$ be functions. Define
 $\phi+\psi\colon X\coprod Y\to  \mathbf{R}$ to be the function that maps
 $x$  in $X$ to $\phi(x)$ and $y$ in $Y$ to $\psi(y)$.
 The  \textbf{coproduct} of two data sets  $\Phi$ and $\Psi$, denoted by $\Phi\coprod\Psi$, is defined to be the data set given by the  measurements
$\{\phi+0\ |\  \phi\in  \Phi\}\cup \{0+\psi\ |\  \psi\in  \Psi\}$ on $X\coprod Y$.  Their  \textbf{product}, denoted by $\Phi\times\Psi$, is defined to be the data set given by the measurements $\{\phi+\psi\ |\  \phi\in  \Phi\text{ and } \psi\in\Psi\}$ on $X\coprod Y$.
The functions:
\[\begin{tikzcd}[column sep=40, row sep=10]
&\Phi\ar[bend left=10]{dr}[sloped,below]{\text{in}_{\Phi}}[sloped,above]{\phi\mapsto\phi+0} &\\
\Phi\times\Psi\ar[bend left=10]{ur}[sloped,below]{\text{pr}_{\Phi}}[sloped,above]{\phi+\psi\mapsto\phi} 
\ar[bend right=10]{dr}[sloped,above]{\text{pr}_{\Psi}}[sloped,below]{\phi+\psi\mapsto\psi}
&& \Phi\coprod \Psi\\
&\Psi\ar[bend right=10]{ur}[sloped,above]{\text{in}_{\Psi}}[sloped,below]{\psi\mapsto 0+\psi} &
\end{tikzcd}\]
satisfy the following universal properties, which justify the names coproduct and  product:
\begin{itemize}
\item for any data set $\Pi$, and any two functions $\alpha\colon \Phi\to \Pi$ and  $\beta\colon \Psi\to \Pi$, there is a unique function
$\mu\colon  \Phi\coprod\Psi\to \Pi$ for which $\mu\, \text{in}_{\Phi}=\alpha$ and $\mu \, \text{in}_{\Psi}=\beta$;
\item for any data set $\Pi$, and any two functions $\alpha\colon \Pi\to \Phi$ and  $\beta\colon \Pi\to\Psi$, there is a unique function
$\mu\colon  \Pi\to \Phi\times\Psi$ for which $\text{pr}_{\Phi}\mu=\alpha$ and $\text{pr}_{\Psi}\mu=\beta$.
\end{itemize}


Let $f\colon\mathbf{R}\to  \mathbf{R}$ be a function. By composing with $f$,  a data set 
$\Phi$ is transformed into a new data set
$f\Phi:=\{f\phi\ |\ \phi\in \Phi\}$. This operation is called   \textbf{change of units} along $f$.  The  symbol $f{-}\colon \Phi\to f\Phi$  denotes the function mapping $\phi$ to $f\phi$.
For example let $f\colon\mathbf{R}\to\mathbf{R}$
 map $\{r\in\mathbf{R}\ |\ r<0\}$ to $-1$ and  $\{r\in\mathbf{R}\ |\ r\geq 0\}$ to $1$. Consider $X=\{x_1,x_2\}$, two data sets $\{1,2\}$ and $\{-1,1\}$ given by the constant functions $-1,1,2\colon X\to \mathbf{R}$, and a function $\alpha\colon\{1,2\}\to \{-1,1\}$ mapping  $1$ to $-1$ and $2$ to $1$.  Then  $f\{1,2\}=\{1\}$ and   $f-\colon \{-1,1\}\to f\{-1,1\}$ is the identity. Thus there is no function $f\{1,2\}\to f\{-1,1\}$ making  the following diagram commutative:
 \[\begin{tikzcd}[column sep=30, row sep=15]
\{1,2\}\ar{r}{f-}\ar{d}[swap]{\alpha}& f\{1,2\}=\{1\}\ar{d}\\
 \{-1,1\}\ar{r}{f-=\text{id}} & f\{-1,1\}=\{-1,1\}
 \end{tikzcd}\]
 Consequently,  for that $f$ there is no functor $F$ assigning to a data set $\Phi$ its change of units $f\Phi$ along $f$   for which $f-\colon\Phi\to f\Phi$ is a natural transformation between $F$ and the identity functor.
 If $f$ is invertible,  then  $f{-}\colon \Phi\to f\Phi$ is a bijection whose inverse is given by $f^{-1}-$.  
The association $(\alpha\colon \Phi\to\Psi)\mapsto \left((f-)\alpha (f^{-1}-)\colon f\Phi\to f\Psi\right)$ is  a functor
 for which $f-\colon \Phi\to f \Phi$ is a natural transformation between this functor and the identity functor.
  Changing the units along  any function  preserves products and coproducts i.e.,
$f(\Phi\coprod \Psi)$ is isomorphic to $f(\Phi)\coprod f(\Psi)$, and $f(\Phi\times \Psi)$ is isomorphic to $f(\Phi)\times f(\Psi)$.
A similar reasoning is used in \cite{giusti15}  to study brain data, in order to obtain results that are invariant under transformations given by change of units with invertible functions, and in \cite{degregorio20} to study metric spaces that are isometric up to a rescaling of the distance functions.

Let $\Phi$ be a data set with the domain $X$. By composing a function $f\colon Y\to X$  with the measurements in $\Phi$, we obtain a new data set $\Phi f:=\{\phi f\ |\ \phi\in \Phi\}$ with the domain $Y$. This operation is called 
 \textbf{domain change} along $f$. 
 The symbol ${-}f\colon \Phi\to \Phi f$ denotes the function that maps $\phi$ to $\phi f$.

Let $f_1\colon Z_1\to X$ and  $f_2\colon Z_2\to Y$ be functions and  $f_1\coprod f_2\colon Z_1\coprod Z_2\to X\coprod Y$
be their coproduct. 
For any  datasets $\Phi$ and $\Psi$ with $\text{dom}(\Phi)=X$ and $\text{dom}(\Psi)=Y$, the following equalities hold:
\[(\Phi\coprod\Psi)(f_1\coprod f_2) = \Phi f_1\coprod\Psi f_2,\ \ \ \ \ \ \ \ (\Phi\times\Psi)(f_1\coprod f_2) = \Phi f_1\times\Psi f_2.\]
 
\section{Metrics and persistent homology}\label{aSFDFGH} 
We can think about a data set $\Phi$ as a subset $\Phi\subset  {\mathbf R}^{|X|}$. Via this inclusion
$\Phi$ inherits a metric induced by the infinity   norm $\| v\|_\infty=\text{max}\{|v_i|\}$  on ${\mathbf R}^{|X|}$. We use the symbol $\|\phi-\psi\|_\infty$
to denote the distance  between $\phi$ and $\psi$ in $\Phi$. The considered data sets are not just sets anymore but metric spaces.  
Therefore non-expansive
($1$-Lipschitz) functions between data sets play a special role. For example, let $f\colon\mathbf{R}\to  \mathbf{R}$ be a function. If $f$ is non-expansive, then so is  the change of units along $f$, 
$f-\colon \Phi\to f\Phi$, that maps $\phi$ to $f\phi$. The domain change $-h\colon\Phi\to \Phi h$ 
is non-expansive along any $h$. Non-expansiveness is an important assumption to prove some stability results in~\cite{DBLP:journals/corr/abs-1812-11832} and it is also reasonable in applications, since it is important that these functions between data sets do not alter the information too much.

By taking all the measurements of $\Phi$ together, we can form  a function
$[\phi_1 \cdots  \phi_m]\colon X\to   {\mathbf R}^{m}$.
Via this function, $X$ inherits a pseudometric $ d_{\Phi}$ induced by the  infinity   norm on ${\mathbf R}^{m}$. Explicitly
$d_{\Phi}(x,y):=\text{max}_{1\leq i\leq m}|\phi_i(x)-\phi_i(y)|$. 
This metric plays a  fundamental role as it permits us 
 to extract persistent homologies (see~\cite{MR2506738,MR2405684}).
In this article,   \textbf{persistent homology} of a data set  $\Phi$ with coefficients in a field  and in a given degree  $d$ assigns a vector space 
$PH_d^{\Phi}(\phi)_{r,s}$ to each measurement $\phi$ in  $\Phi$,  for every $(r,s)$ in $[0,\infty)\times  {\mathbf R}$, and it is defined as:
\[PH_d^{\Phi}(\phi)_{r,s}:=H_d\left( \text{VR}_r (  \phi\leq s,d_{\Phi}   )  \right)\text{, where:}\]
\begin{itemize}
\item $\phi\leq s:=\phi^{-1}(-\infty,s]$;
\item $\text{VR}_r\left(\phi\leq s,d_{\Phi}\right)$ is the \textbf{Vietoris-Rips} complex whose simplices are given by the subsets  $\sigma\subset (\phi\leq s)$ of  diameter not exceeding $r$ with respect to $d_{\Phi}$;
\item $H_d$ is the homology in degree $d$ with coefficients in a given field.
\end{itemize}

If $s\leq s'$ and $r\leq r'$, then $(\phi\leq s)\subset (\phi\leq s')$ and   therefore $\text{VR}_r(\phi\leq s)\subset \text{VR}_{r'}(\phi\leq s')$.  The   linear function  induced  on homology by this inclusion  is denoted by:
\[PH_d^{\Phi}(\phi)_{(r,s)\leq(r',s')}\colon PH_d^{\Phi}(\phi)_{r,s}\to PH_d^{\Phi}(\phi)_{r',s'}.\]
These functions form a functor $PH_d^{\Phi}(\phi)$ indexed by the poset $[0,\infty)\times {\mathbf R}$ with values in the category of vector spaces. Since $X$ is finite,    $PH_d^{\Phi}(\phi)$ is \textbf{tame} (see~\cite{MR3735858}). This means that   values of $PH_d^{\Phi}(\phi)$  are finite dimensional, and there are two finite  sequences $0=r_0<r_1<\cdots <r_m$ in $[0,\infty)$  and $s_0<s_1<\cdots <s_l=\infty$ in  ${\mathbf R}$  such that
 $PH_d^{\Phi}(\phi)$, restricted to  subposets of the form $ [r_i,r_{i+1})
\times (\infty,s_0)\subset  [0,\infty)\times {\mathbf R}$ and $[r_i,r_{i+1})\times [s_j,s_{j+1})\subset [0,\infty)\times {\mathbf R}$, is constant.
 The category of such functors is denoted by  $\text{Tame}( [0,\infty)\times  {\mathbf R},\text{Vect})$.
 Thus a data set $\Phi$ leads to a function  assigning to each measurement  $\phi$  its persistent homology in a given degree:
 \[PH_d^{\Phi}\colon \Phi\to\text{Tame}( [0,\infty)\times  {\mathbf R},\text{Vect}).\]
 
 Next,   recall a  definition of the \textbf{interleaving metric} in the direction of the vector $(0,1)$
 on  $\text{Tame}( [0,\infty)\times  {\mathbf R},\text{Vect}) $ (see~\cite{MR3348168}).  Let $P$ and $Q$ be in  $\text{Tame}( [0,\infty)\times  {\mathbf R},\text{Vect}) $.
 \begin{itemize}
 \item 
$P$ and $Q$  are $\epsilon$-\textbf{interleaved} if, for all    $(r,s)$ in  $[0,\infty)\times {\mathbf R}$, there are linear functions $f_{s,r}\colon P_{r,s}\to Q_{r,s+\epsilon}$ and $g_{s,r}\colon  Q_{r,s}\to P_{r,s+\epsilon}$ making the following
 diagram commutative:
 \[\begin{tikzcd}
 & P_{r,s}\ar{dr}{f_{s,r}}\ar{rr}{P_{(r,s)<(r,s+2\epsilon)}}  & & P_{r,s+2\epsilon}\ar{rd}{f_{r,s+2\epsilon}}\\
 Q_{r,s-\epsilon}\ar{ur}{g_{r,s-\epsilon}}\ar{rr}[swap]{Q_{(r,s-\epsilon)<(r,s+\epsilon)}} & &  Q_{r,s+\epsilon} \ar{rr}[swap]{Q_{(r,s+\epsilon)<(r,s+3\epsilon)}}
 \ar{ur}{g_{r,s+\epsilon}}& &  Q_{r,s+3\epsilon}
 \end{tikzcd}\]
 \item  $d_{\bowtie}(P,Q):=\text{inf}\{\epsilon\in [0,\infty)\ |\ \text{$P$ and $Q$ are $\epsilon$-interleaved}\}$.
 \end{itemize}
 The function $P,Q\mapsto d_{\bowtie}(P,Q)$ is an extended ($\infty$ is allowed)  metric on the set $\text{Tame}( [0,\infty)\times  {\mathbf R},\text{Vect}) $
 called   interleaving metric in the direction of the vector $(0,1)$.
 \begin{proposition}
The  function $PH_d^{\Phi}\colon \Phi\to\text{\rm Tame}( [0,\infty)\times  {\mathbf R},\text{\rm Vect})$
 is non-expansive if the set $ \Phi$ is equipped with  $\infty$-norm metric $\|\phi-\psi\|_\infty$ and the set 
 $\text{\rm Tame}( [0,\infty)\times  {\mathbf R},\text{\rm Vect})$ is equipped with the interleaving metric in the direction of the vector $(0,1)$.
 \end{proposition}
 \begin{proof}
 Let $\phi,\psi\colon X\to \mathbf{R}$ be measurements in $\Phi$ and $\epsilon=\|\phi-\psi\|_\infty$. For every  $s$ in $\mathbf{R}$, the sublevel set $\phi\leq s$
 is a subset of $\psi\leq s+\epsilon$, and  $\psi\leq s$ is a subset of $\phi\leq s+\epsilon$. This translates into
 inclusions: \[ \text{VR}_r (  \phi\leq s,d_{\Phi}   )\subset  \text{VR}_r (  \psi\leq s+\epsilon,d_{\Phi}   )\ \ \ \ \  \text{VR}_r (  \psi\leq s,d_{\Phi}   )\subset  \text{VR}_r (  \phi\leq s+\epsilon,d_{\Phi}   )\] leading
 functions:
 \[f_{s,r}\colon PH_d^{\Phi}(\phi)_{r,s}\to PH_d^{\Phi}(\psi)_{r,s+\epsilon}\ \ \ \ \ g_{s,r}\colon PH_d^{\Phi}(\psi)_{r,s}\to PH_d^{\Phi}(\phi)_{r,s+\epsilon}.\]  These functions provide $\epsilon$ interleaving between $PH_d^{\Phi}(\phi)$ and $PH_d^{\Phi}(\psi)$, giving   $\|\phi-\psi\|_\infty\geq d_{\bowtie}(PH_d^{\Phi}(\phi),PH_d^{\Phi}(\psi))$.
  \end{proof}
 
 A measurement $\phi\colon X\to \mathbf{R}$ can be part of many data sets and its persistent homology  depends on what data set this function
 is part of.  For example,  let $X=\{x_1,x_2,x_3,x_4\}$ and $\phi,\psi\colon X\to \mathbf{R}$  be measurements defined as follows:
 \[\begin{array}{ccc}
  \phi(x_1)=-1 &  \phi(x_2)=\phi(x_3)=0 & \phi(x_4)=1  \\
  \hline
   \psi(x_3)=-1 &  \psi(x_1)= \psi(x_4)=0 &  \psi(x_2)=1
   \end{array}\]
The measurement $\phi$ is   part of two data sets $\Phi=\{\phi\}$ and $\Psi=\{\phi,\psi\}$.
The induced  pseudometrics $d_{\Phi}$ and  $d_{\Psi}$ on $X$ can be depicted by the following diagrams where
the continuous,  dashed, and dotted lines indicate distance $0$, $1$ and $2$ respectively:
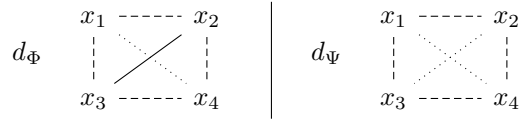
\[\begin{array}{rl|rl}
d_\Phi & \begin{tikzcd}
x_1\ar[dash, dashed]{r} \ar[dash, dashed]{d}\ar[dash,dotted]{dr} & x_2 \ar[dash, dashed]{d} \ar[dash]{dl} \\
x_3 \ar[dash, dashed]{r} & x_4
\end{tikzcd}\ \ \ 
&
\ \ \ d_\Psi
&
\begin{tikzcd}
x_1\ar[dash, dashed]{r} \ar[dash, dashed]{d}\ar[dash,dotted]{dr} & x_2 \ar[dash, dashed]{d} \ar[dash, dotted]{dl} \\
x_3 \ar[dash, dashed]{r} & x_4
\end{tikzcd}
\end{array}
\]
In this case $PH_1^{\Phi}(\phi)_{r,s}=0$ for all $r$ and $s$, however:
\[\text{dim}PH_1^{\Psi}(\phi)_{r,s}=\begin{cases}
1 &\text{ if } 1\leq s \text{ and } 1\leq r<2\\
0 &\text{ otherwise } 
\end{cases}\]


 To understand persistent homology, it is therefore paramount to understand  how it changes when  data sets change and here  functoriality plays an essential role.

 Let $\Phi$ and  $\Psi$ be data sets  consisting of measurements on $X$ and  $Y$ respectively.
 A function $\alpha\colon \Phi\to \Psi$ is called \textbf{geometric} if there is a function $f\colon Y\to X$, called a \textbf{realization} of $\alpha$,  making  the following diagram commutative for every $\phi$ in $\Phi$:
  \[\begin{tikzcd}[row sep=3]
Y\ar{dd}[swap]{f}\ar{dr}[pos=.3]{\alpha(\phi)} \\
& \mathbf{R}\\
 X\ar{ru}[swap,pos=.3]{\phi}
 \end{tikzcd}\]
 For example $-f\colon \Phi\to \Phi f$ is geometric, as it is realized by $f$.
 
The commutativity of the   triangle  above has two consequences. First, $f$ is non-expansive with respect to the pseudometrics  $d_{\Phi}$ on  $X$ and $d_{\Psi}$ on $Y$. Second,   for  $s$ in $\mathbf{R}$ and  $\phi$ in $\Phi$, the subset  $(\alpha(\phi)\leq s)\subset Y$ is mapped via $f$ into
$(\phi\leq s)\subset X$, i.e., the following diagram commutes:
\[\begin{tikzcd}[row sep=3]
\alpha(\phi)\leq s\ar[hook]{r}\ar{dd}[swap]{f} & Y\ar{dd}[swap]{f}\ar{dr}[pos=.3]{\alpha(\phi)}\\
& & \mathbf{R}\\
\phi\leq s\ar[hook]{r} & X\ar{ur}[swap,pos=.3]{\phi}
\end{tikzcd}\]
The  realization $f$  induces therefore  a map of Vietoris-Rips complexes and their homologies:
\[f_{s,r}\colon \text{VR}_r(\alpha(\phi)\leq s,d_{\Psi} )\to \text{VR}_r(\phi\leq s,d_{\Phi} );
\]
\[\begin{tikzcd}[column sep=50, row sep =10]
PH_d^{\Psi}(\alpha(\phi))_{r,s}\ar[equal]{d} &PH_d^{\Phi}(\phi)_{r,s} \ar[equal]{d}\\
 H_d\left( \text{VR}_r (\alpha(\phi)\leq s ,d_{\Psi}   )  \right)
\ar{r}{H_d(f_{r,s})} &
H_d\left( \text{VR}_r (\phi\leq s ,d_{\Phi}   )  \right).
\end{tikzcd}
\]

  If $f,f'\colon Y\to X$ are two realizations of $\alpha$, then for  $y$ in $Y$,
  $d_{\Phi}(f(y),f'(y))=0$, hence they are points of the same simplex in the Vietoris-Rips complex, implying that $f_{r,s}$ and $f'_{r,s}$ are homotopic  for all $r$ and $s$. 
  Consequently, $H_d(f_{r,s})=H_d(f'_{r,s})$. The linear function
  $H_d(f_{r,s})$ depends therefore only on  $\alpha$ and it is independent on the choice of its realization $f$.
 We  denote this function by:
 \[PH_d^{\alpha}(\phi)_{r,s}\colon PH_d^{\Psi}(\alpha(\phi))_{r,s}\to     PH_d^{\Phi}(\phi)_{r,s}.\]
 These functions are  natural in $r$ and $s$  and induce  a morphism in the category $\text{Tame}( [0,\infty)\times  {\mathbf R},\text{Vect})$
 between persistent homologies:
 \[PH_d^{\alpha}(\phi)\colon PH_d^{\Psi}(\alpha(\phi))\to  PH_d^{\Phi}(\phi).\]

 If  $\alpha\colon\Phi\to\Psi$ and $\beta\colon\Psi\to \Xi$ are geometric functions realized by  $f\colon Y\to X$ and $g\colon Z\to Y$,
 then the composition $\beta\alpha \colon \Phi\to  \Xi$ is also geometric, and realized by the composition $fg\colon Z\to X$.  Consequently, for every 
 measurement $\phi$ in $\Phi$,
 $PH_d^{\beta\alpha}(\phi)=PH_d^{\alpha}(\phi)PH_d^{\beta}(\alpha(\phi))$, that assures the commutativity of the diagram:
 \[\begin{tikzcd}[column sep=40]
 PH_d^{\Xi}(\beta\alpha(\phi)) \ar{r}{PH_d^{\beta}(\alpha(\phi))}\ar[bend right=15]{rr}[swap]{PH_d^{\beta\alpha}(\phi)}
  &PH_d^{\Psi}(\alpha(\phi))\ar{r}{PH_d^{\alpha}(\phi)} &  PH_d^{\Phi}(\phi)
 \end{tikzcd}\]

For any $\alpha\colon \Phi\to \Psi$,  taking persistent homology leads to two functions on $\Phi$:
\[\begin{tikzcd}[row sep=4, column sep=28]
 &  &\text{Tame}( [0,\infty)\times  {\mathbf R},\text{Vect})\\
 \Phi\ar[bend left=11,end anchor=west]{urr}{PH_d^\Phi} \ar[bend right=15,end anchor=west]{rd}{\alpha}\\
  & \Psi\ar{r}{PH_d^\Psi} & \text{Tame}( [0,\infty)\times  {\mathbf R},\text{Vect})
 \end{tikzcd}\]
These functions  rarely  coincide. 
 However, when  $\alpha$ is geometric,  we can use the morphisms $PH_d^{\alpha}(\phi)\colon PH_d^{\Psi}(\alpha(\phi))\to  PH_d^{\Phi}(\phi)$  to compare the values of these two functions on  $\Phi$. For non-geometric $\alpha$, we are not equipped with such comparison morphisms and there is no reason for such a comparison to even exist. For example,
 consider the change of unit along the  function $f\colon\mathbf{R}\to \mathbf{R}$, $ f(x):=-x$. 
 Then $f-\colon \Phi\to f\Phi$ is an isomorphism. 
 In this case 
 \[\begin{array}{l|r}
 PH_d^{\Phi}(\phi)_{r,s}:=H_d\left( \text{VR}_r (  \phi\leq s,d_{\Phi}   )  \right)  & 
 (f-)PH_d^{f\Phi}(\phi)=H_d\left( \text{VR}_r (  \phi\geq -s,d_{\Phi}   )  \right).
 \end{array}\]
 Thus $PH_d^{\Phi}$ encodes information about sub-level sets of the measurements in $\Phi$ and 
$ (f-)PH_d^{f\Phi}$ encodes information about super-level sets  of the measurements. These persistent  homologies encode
therefore 
the same information as the  so called extended persistence (see~\cite{MR2472288,MR3408277}).
 \section{Actions}\label{asfadfhg}
To describe symmetries of a data set $\Phi$ with domain $X$, we  consider operations on $X$ that convert    measurements into  measurements. By definition a $\Phi$-\textbf{operation} is a function $g\colon X\to X$ such that, for every measurement $\phi$ in $\Phi$,  the composition 
$\phi g$   also belongs to $\Phi$.  If $g\colon X\to X$ is such an operation, then, for all $\phi$ and $\psi$ in $\Phi$:
\[\|\phi-\psi\|_{\infty}=\text{max}_{x\in X}|\phi(x)-\psi(x)|\geq
	\text{max}_{x\in \text{im}(g)}|\phi(x)-\psi(x)|=\|\phi g-\psi g\|_{\infty}.\]
	Thus the function $-g\colon \Phi\to \Phi$ that maps $\phi$ to $\phi g$ is non-expansive.

The composition of $\Phi$-operations is again a $\Phi$-operation, and the identity function $\text{id}_X$ is also a $\Phi$-operation. In this way the set of $\Phi$-operations with the composition becomes a unitary monoid, called the \textbf{structure monoid} of $\Phi$,  and  denoted by: 
\[ \text{End}_{\Phi}(X)=\{g:X\to X\mid \phi  g\in \Phi  \text{ for every }\phi\in\Phi\}\subset \text{End}(X).\] 
A $\Phi$-operation $g$ is invertible if there is a $\Phi$-operation $h$ such that $gh=hg=\text{id}_X$. Since $\Phi$ is finite,  a $\Phi$-operation is invertible   if and only if it  is a bijection.  Their collection is denoted by:
\[ \text{Aut}_{\Phi}(X)=\{g:X\to X\mid  \text{ $g$ is a bijection, and } \phi  g\in \Phi  \text{ for every }\phi\in\Phi\}.\]
With the composition operation, $\text{Aut}_{\Phi}(X)$  becomes a group for which  the inclusion
$\text{Aut}_{\Phi}(X)\subset \text{End}_{\Phi}(X)$ is a monoid homomorphism. 

A data set $\Phi$ is equipped with  an associative  right action:
\[\Phi\times \text{End}_{\Phi}(X)\to \Phi,\ \ \ \ (\phi,g)\mapsto \phi g.\]
Thus $\Phi$ is not just a set, but a set with an action of the  monoid $\text{End}_{\Phi}(X)$.   To encode the symmetries of $\Phi$ 
induced by this action, we  consider its  incarnations.

An \textbf{incarnation} of   $\Phi$ is a choice of  a subset $M\subset \text{End}_\Phi(X)$ (in general, not necessarily a submonoid).  An incarnation is denoted as a pair $(\Phi,M)$.  We think about $M$ as an additional structure on $\Phi$. An incarnation of the form $(\Phi,M)$  is called an $M$-incarnation. We also refer to an $M$-incarnation as an $M$-action. The choice of an $M$-action on $\Phi$ encodes certain  symmetries of $\Phi$.
Different  choices of $M$ can encode different symmetries. This flexibility is important in applications. For example in data sets that represent images, we might want to focus on rotational symmetries, so we may use an appropriate action on the data set to inject the corresponding geometry.
The incarnation $(\Phi,\text{End}_{\Phi}(X))$ is an example of 
a incarnation called  universal.

An incarnation  $(\Phi,M)$ is called a  \textbf{monoid incarnation} if $M\subset \text{End}_\Phi$  is a submonoid, and our convention here is that all such submonoids contain the identity element.  If  $(\Phi,M)$ is an incarnation, we use the symbol $(\Phi,\langle M\rangle)$ to denote the
monoid incarnation where $\langle M\rangle\subset \text{End}_\Phi(X)$ is the submonid generated by $M$.

If a submonoid  $M\subset \text{End}_\Phi(X)$ is a group, then $(\Phi,M)$ is called a  \textbf{group incarnation}.
The incarnation $(\Phi, \text{Aut}_{\Phi}(X))$ is an example of a group incarnation called   universal.

Let $(\Phi,M)$ be an incarnation for which any element $g$ in  $M$ is a bijection.
Such incarnations are called \textbf{group-like}.
For group like incarnations  $(\Phi,M)$ the  finiteness  implies that  the monoid $\langle M \rangle $ is in fact a subgroup of $\text{Aut}_{\Phi}(X)$. 
Thus  any group-like incarnation $(\Phi,M)$ leads to a group incarnation $\left(\Phi, \langle M \rangle\right) $.
 
 Let $(\Phi,M)$ be an incarnation.
 For a subset  $\Omega\subset \Phi$,  the symbol  $ \Omega M$ denotes the set of all the  measurements in $\Phi$
which either belong to  $\Omega$ or are  of the form $\omega g_1\cdots g_k$, for some $\omega$ in $\Omega$ and some sequence 
of elements $g_1,\ldots g_k$  in $M$.  If $\Omega M=\Phi$,  then $\Omega$ is said to \textbf{generate} the incarnation  $(\Phi,M)$.
 In the case  $(\Phi,M)$ is a monoid incarnation, then  any element in $ \Omega M$  is of the form  $\omega g$ for some
 $\omega$ in $\Omega$  and $g$ in $M$.   
 Note that $\Omega M= \Omega\langle M\rangle$ for every incarnation $(\Phi,M)$.

 
 If $\psi$ belongs to $\phi M:=\{\phi\}M$, then $\psi$ is said to be a \textbf{deformation} of   $\phi$. 
  If $(\Phi,M)$ is a group incarnation, then the relation of being a deformation is an equivalence relation. For a general incarnation however being a deformation can fail to be even a symmetric relation. 
  Two measurements in $\Phi$ are said to be \textbf{connected}  if they are related by the equivalence relation generated by the relation of being a deformation. The symbol $\Phi/M$ denotes the partition of $\Phi$  induced by  this equivalence relation. We refer to  $\Phi/M$ as the \textbf{quotient} of the incarnation $(\Phi,M)$. 
The partitions $\Phi/M$ and $\Phi/\langle M\rangle$ coincide. 
   If $(\Phi,M)$ is a group incarnation, then  $\Phi/M$ coincide with the orbit partition of the usual group action of $M$ on  $\Phi$.

 Let $(\Phi,M)$ be an incarnation.
  For a measurement  $\psi$ in $\Phi$, the symbol $[\psi ]$  denotes the block in  $\Phi/M$ containing $\psi$. Explicitly,  $[\psi ]$ is the subset of $\Phi$ consisting of all the measurements connected to $\psi$. Note that, for all $g$ in $M$, if $\phi$ is connected to $\psi$, then 
   $\phi g$ is also connected to $\psi$. We thus have the following  inclusions:
   \[\begin{tikzcd}
   M\ar[hook]{r} \ar[hook']{d} & \text{End}_{\Phi}(X)\ar[hook]{d}\\
    \text{End}_{[\psi]}(X)\ar[hook]{r} & \text{End}(X)
   \end{tikzcd}\]
 The $M$ incarnation $([\psi],M)$ of the block $[\psi]$, given by the above  inclusions $M\subset  \text{End}_{[\psi]}$, is called a \textbf{block incarnation} 
 of $(\Phi,M)$. In this way we can think about $[\psi]$ and $([\psi],M)$ as a new data set.

    

An incarnation $(\Phi,M)$  is called  \textbf{transitive} if all the elements in $\Phi$ are connected to each other. For example, let 
$M$ be a finite submonoid of  $\text{End}(X)$.  For a given function $\phi\colon X\to\mathbf{R}$, define a data set
$\phi M:=\{\phi g\  |\ g\in M\}$ to consist of all functions of the form  $x\mapsto \phi(g(x))$ for all $g$ in $M$.  Then  every $g\colon X\to X$ in $M$  is a $\phi M$-operation.  The obtained incarnation $(\phi M, M)$ is transitive. Any transitive group incarnation is  of  such form. 
For all measurements $\phi$ in any incarnation  $(\Phi,M)$, the block incarnation 
$([\phi],M)$ is transitive.  Any transitive incarnation is  of this form.


 Let $(\Phi,M)$ be an   incarnation. 
A subset  $\Omega\subset \Phi$ is called \textbf{independent} if  no element in $\Omega$  is a deformation of any other element in $\Omega$,
explicitly:
$\omega\not\in \omega'M$ for all  $\omega\not=\omega'$ in $\Omega$.

A \textbf{basis} of $(\Phi,M)$ is an independent  subset  $\Omega\subset \Phi$  such that $\Omega M=\Phi$ ($\Omega$ generates $(\Phi,M)$). 

Two measurements $\psi$ and $\phi$ are called  \textbf{indistinguishable}  if $\psi$ is a deformation of $\phi$ and $\phi$ is a deformation of $\psi$.
 If $(\Phi,M)$ is a group incarnation, then $\psi$ and $\phi$ are indistinguishable if and only if $\psi=\phi g$ for some $g$ in $M$, i.e., if $\psi$ is a deformation of $\phi$.   
 
 \begin{proposition}\label{sasfdfshgfd}
\begin{enumerate}
\item Every incarnation  has a basis.
\item 
Let $\Omega,\Omega'\subset \Phi$ be two bases of an incarnation  $(\Phi,M)$. Then there is a bijection $\sigma\colon\Omega\to \Omega'$
such that $\omega$ and $\sigma(\omega)$ are indistingishable for every $\omega$ in $\Omega$.
\end{enumerate}
\end{proposition}
\begin{proof}
\noindent
(1):\quad  Let $(\Phi,M)$ be an incarnation. Choose $\Omega\subset  \Phi$ to be an  independent  subset for which $\Omega M$ is maximal.
Existence of $\Omega$ is guaranteed by finiteness of $\Phi$.
We claim that $\Omega M=\Phi$ and hence $\Omega$ is a basis. If this is not the case, let $\psi$ be in $\Phi\setminus\Omega M $. Define $\Omega'=\{\psi\}\cup\{\omega\in \Omega\ |\ \omega\not\in \{\psi\}M\}$.  Then $\Omega'M$ contains $\Omega$ and hence $\Omega M$. It also contains
$\psi$.  Since    $\Omega' $ is independent, we would obtain a   contradiction  to the maximality assumption about $\Omega M$, and thus the claim holds.
\smallskip

\noindent
(2):\quad
Let $\omega$ be in $\Omega$.
Since $\Omega M=\Phi=\Omega' M$, there is  $\omega'$ in $\Omega'$ such that $\omega\in \omega'M$.
Let $\omega_1$ in $\Omega$ be such that $\omega'\in  \omega_1M$.  
Then $\omega\in  \omega'M\subset  \omega_1M$, and hence  $\omega=\omega_1$ by the  independence of $\Omega$.
The desired bijection is then  given by $\omega\mapsto \omega'$.
\end{proof}

According to Proposition~\ref{sasfdfshgfd},  any two bases of an incarnation have the same number of elements. We define the \textbf{dimension} of an incarnation to be the cardinality of its bases. For example a transitive group incarnation has dimension $1$. In fact  for a transitive group incarnation any  single measurement  forms a basis. More generally, the dimension
of a group incarnation $(\Phi,M)$ equals the cardinality of $\Phi/M$. In this case   $\Omega\subset \Phi$ is a basis if and only if,
for every   block $\Psi$ in  
 $\Phi/M$, the intersection $\Omega\cap \Psi$ has only one element.
 Since being a basis depends only on the monoid $\langle M\rangle$, the dimension of a group-like incarnation $(\Phi,M)$ equals  also the  cardinality of $\Phi/M$, and similarly a subset   $\Omega\subset \Phi$ is a basis if and only if, for every   block $\Psi$ in the partition $\Phi/M$, the intersection $\Omega\cap \Psi$ has only one element.

The dimension of a transitive monoid incarnation can be  bigger than $1$.  For example,  let $X= \{x_1,x_2,x_3\}$ and consider 
  functions $\phi_1,\phi_2,\phi_3\colon X\to\mathbf{R}$ and  $g_1,g_2,g_3\colon X\to X$ defined as follows:
\[{\begin{array}{c|c|c|c|c|c}
\phi_1(x_1)=2 &\phi_2(x_1)=2  &   \phi_3(x_1)=1 &g_1(x_1)=x_2 & g_2(x_1)=x_2&g_3(x_1)=x_1  \\
\phi_1(x_2)=2 &\phi_2(x_2)=2&   \phi_3(x_2)=2 & g_1(x_2)=x_2& g_2(x_2)=x_2 & g_3(x_2)=x_2\\
 \phi_1(x_3)=3 &\phi_2(x_3)=2 &  \phi_3(x_3)=2 & g_1(x_3)=x_3& g_2(x_3)=x_2 & g_3(x_3)=x_2
\end{array}}
\]
The compositions $g_ig_j$ and $ \phi_i g_j$ are described by the following tables:
\[{\begin{array}{c|c|c|c}
 & g_1 & g_2 & g_3 \\
 \hline
 g_1 & g_1&g_2 & g_2\\
  \hline
 g_2 & g_2 & g_2 & g_2\\
  \hline
 g_3 &g_2 & g_2 & g_3
\end{array}\quad \quad\quad
\begin{array}{c|c|c|c}
 & g_1 & g_2 & g_3 \\
  \hline
 \phi_1 & \phi_1 &  \phi_2 &  \phi_2\\
   \hline
 \phi_2 & \phi_2 & \phi_2 & \phi_2\\
  \hline
 \phi _3 &\phi_2 & \phi_2 & \phi_3
\end{array}
}
\]
Thus the functions $g_1$, $g_2$, and $g_3$ are  $\Phi:=\{ \phi_1,\phi_2,\phi_3\}$-operations.
Furthermore the subset $M:=\{\text{id}, g_1,g_2,g_3\}\subset \text{End}_{\Phi}(X)$ is a submonoid. 
The incarnation $(\Phi,M)$ is a transitive monoid incarnation.  Since  the set $\{\phi_1, \phi_3\}$
is independent and generates $(\Phi,M)$, it is a basis. Thus $(\Phi,M)$  is an example of a transitive monoid incarnation of dimension $2$.
\section{Nirvana}\label{dfadfhbsfg}
To compare  incarnations of various data sets we are going to use  SEOs (set equivariant  operators). 
A \textbf{SEO} from an incarnation $(\Phi,M)$ to an incarnatiopn $(\Psi,N)$, denoted as $(\alpha,T)\colon (\Phi,M)\to (\Psi,N)$,  is a pair of functions
$(\alpha\colon \Phi\to \Psi, T\colon M\to N)$ for which the following diagram commutes:
\[\begin{tikzcd}[column sep=30]
  \Phi\times M \ar[hook]{r}\ar{d}[swap]{\alpha\times T}  & \Phi\times \text{End}_{\Phi}(X)\ar{r}{\text{action}} & \Phi\ar{d}{\alpha}\\
 \Psi\times N\ar[hook]{r} &  \Psi\times  \text{End}_{\Psi}(Y)\ar{r}{\text{action}} & \Psi
 \end{tikzcd}
 \]
  Explicitly,  for $\phi$ in $\Phi$ and  $g$ in $M$, it holds $\alpha(\phi g) = \alpha(\phi)T(g)$. This implies that for $\phi$ in $\Phi$ and 
  every sequence of elements  $g_1,\ldots,g_k$ in $M$, it holds: 
  \[\alpha(\phi g_1\cdots g_k)=\alpha(\phi) T(g_1)\cdots T(g_k).\]
  Be however aware that in general  there may not be a  homomorphism $T\colon\langle M\rangle\to \langle N\rangle$ of monoids which extends $T\colon M\to N$ and makes the following diagram commutative:
  
  \[\begin{tikzcd}[column sep=30]
  \Phi\times M \ar[hook]{r}\ar{d}[swap]{\alpha\times T}  & \Phi\times \langle M\rangle\ar[hook]{r}\ar{d}[swap]{\alpha\times T}  & \Phi\times \text{End}_{\Phi}(X)\ar{r}{\text{action}} & \Phi\ar{d}{\alpha}\\
 \Psi\times N\ar[hook]{r} & \Psi\times \langle N\rangle\ar[hook]{r}&  \Psi\times  \text{End}_{\Psi}(Y)\ar{r}{\text{action}} & \Psi
 \end{tikzcd}
 \]

 A SEO between monoid incarnations $(\alpha,T)\colon (\Phi,M)\to (\Psi,N)$ is called a MEO (monoid equivariant  operators) if $T\colon M\to N$ is a monoid homomorphism.
 A MEO between group incarnations is also called a GEO (group equivariant  operators).



Let $(\alpha_0,T_0)\colon (\Phi_0,M_0)\to (\Phi_1,M_1)$  and $(\alpha_1,T_1)\colon (\Phi_1,M_1)\to (\Phi_2,M_2)$
 be SEOs. Then  the compositions $(\alpha_1\alpha_0,T_1T_0)$ form    a SEO. Furthermore the pair $(\text{id}_\Phi,\text{id}_M)\colon (\Phi,M)\to(\Phi, M)$ is also  a SEO.  
 The composition of SEOs is an associative   operation and defines a category structure on the collection of data set incarnations with 
 SEOs as morphisms.  This category is called \textbf{Nirvana}. 

 A SEO  $(\alpha,T)\colon (\Phi,M)\to (\Psi,N)$ is an isomorphism if and only if both of the functions $\alpha$ and $T$ are bijections.
 Isomorphisms preserve independence and being a basis:

 \begin{proposition}\label{asfafdgsfgb}
 If   $(\alpha,T)\colon (\Phi,M)\to (\Psi,N)$  is  an isomorphism, then
 a subset $\Omega\subset\Phi$ is independent or a basis if and only if its image $\alpha(\Omega)\subset \Psi$ is  independent or a basis.
 \end{proposition}
 \begin{proof}
 Assume $\alpha$ and $T$ are bijections. This assumption imply that $\phi_1$ belongs to $\phi_2 M$ if and only if
 $\alpha(\phi_1)$ belongs to $\alpha(\phi_2)N$.  It follows that two elements in $\Phi$ are (in)dependant if and only if 
 their images via $\alpha$ are (in)dependent in $\Psi$.   By the  same argument,  $\Omega M = \Phi$ if and only $\alpha(\Omega) T(M)=\alpha(\Phi)$.
 \end{proof}
  According to Proposition~\ref{asfafdgsfgb}  two isomorphic incarnations have the same dimension.
  
  The universal incarnations $(\Phi, \text{End}_{\Phi}(X))$ and  $(\Phi,\text{Aut}_{\Phi}(X))$ are special in the category Nirvana. 
 For any  $(\Phi,M)$, the pair $(\text{id}, i\colon M\hookrightarrow \text{End}_{\Phi}(X))$ defines a SEO
 $(\Phi,M)\to (\Phi,\text{End}_{\Phi}(X))$ called \textbf{canonical}. 
 If $(\Phi,M)$  is a group incarnation, 
 then the pair
 $(\text{id}, i\colon M\hookrightarrow \text{Aut}_{\phi}(X))$
 defines a GEO $(\Phi,M)\to (\Phi,\text{Aut}_{\phi}(X))$ also called canonical.

 
 The rest of this section is devoted to present three ways of constructing SEOs.
 \smallskip
 
 \noindent
 \textbf{Change of units.}\quad 
   Choose  a function $f\colon\mathbf{R}\to \mathbf{R}$. For any incarnation   $(\Phi,  M)$,  consider the data set $f\Phi$ (see Section~\ref{asfdsfghfw}).
 If $g$ is a $\Phi$-operation, then
it is also a $f\Phi$-operation. Thus there is an  inclusion $\text{End}_{\Phi}(X)\subset \text{End}_{f\Phi}(X)$, which
is an equality if  $f$ is invertible, therefore we have an incarnation $(f\Phi,M)$.
If $(\Phi,  M)$ is  a monoid or a group incarnation, then so is $(f\Phi,M)$.
The pair $(f-,\text{id}_M)\colon (\Phi,M)\to (f\Phi,M)$ is a SEO called the change of units along $f$.

Assume $f$ is invertible. 
If $(\alpha,T)\colon (\Phi,M)\to (\Psi,N)$ is a SEO, then the pair of functions $\left((f-)\alpha(f^{-1}-),T\right)$ forms a SEO between
 $(f\Phi,M)$  and $(f\Psi,N)$. The assignment $(\alpha,T)\mapsto ((f-)\alpha(f^{-1}-),T)$  is a self functor $\mathrm{C}(f)$ of Nirvana
 also called the change of units along $f$. It is an equivalence of categories. 
 Indeed,
  \begin{equation*}
     \begin{split}
         \mathrm{C}(f)\mathrm{C}(f^{-1})((\Phi,M)) & = \mathrm{C}(f)(f^{-1}\Phi, M)=(\Phi, M)\\
         \mathrm{C}(f)\mathrm{C}(f^{-1})((\alpha,T)) &= \mathrm{C}(f)((f^{-1}-)\alpha(f-),T)) \\
         &= ((f-)(f^{-1}-)\alpha(f-)(f^{-1}-),T) = (\alpha,T).
     \end{split}
 \end{equation*}
 The same holds for $\mathrm{C}(f^{-1})\mathrm{C}(f)$, hence $\mathrm{C}(f)$ is an equivalence of categories.
 The SEOs  $(f-,\text{id}_M)\colon (\Phi,M)\to (f\Phi,M)$, for all incarnations $(\Phi,M)$,  form a natural transformation between the identity functor on \textbf{Nirvana} and the change of units along $f$ functor.
  \smallskip
 
 \noindent
 \textbf{Domain change.}\quad 
 Let  $(\Phi,M)$ and  $(\Psi,N)$ be incarnations of data sets  consisting of measurements on $X$ and  $Y$ respectively.
 A SEO $(\alpha,T)\colon (\Phi,M)\to (\Psi,N)$ is called \textbf{geometric} if there is a function $f\colon Y\to X$, called a realization of $(\alpha,T)$,   making the following diagram commutative for every $\phi$ in $\Phi$ and $g$ in $M$:
  \[\begin{tikzcd}[row sep=4]
Y\ar{r}{T(g)}\ar{dd}[swap]{f} & Y\ar{dr}[pos=0.3]{\alpha(\phi)}\ar{dd}[swap]{f}
 \\
 & & \mathbf{R}
 \\
 X\ar{r}[swap]{g} & X \ar{ur}[swap,pos=0.3]{\phi}
 \end{tikzcd}\]
 For example, let $(\Phi,M)$ be an  incarnation of a    data set consisting of
 measurements on $X$. Then the SEO $(\text{id}_\Phi, \text{id}_M)\colon (\Phi,M)\to (\Phi,M)$ is geometric. The identity function $\text{id}_X\colon X\to X$ is one of  its realizations. 

Let $Y\subset X$ be  $M$-invariant: $g(y)$  belongs to $Y$ for all $y$ in $Y$ and $g$ in $M$.
Consider   the data set   $\Phi|_{Y}$ given by the domain change along the  inclusion $Y\subset X$.
The  restriction of $g$ to $Y$ is a $\Phi|_{Y}$-operation for every $g$ in $M$. We use the symbol  $T_Y\colon M\to \text{End}_{\Phi|_{Y}}(Y)$
to denote  the function that maps $g$ in $ M$ to the restriction of $g$ to $Y$.  The incarnation $(\Phi|_{Y}, T_Y(M))$ is called the \textbf{restriction} of
$(\Phi,M)$ to the invariant subset $Y$.
The pair $(\Phi\twoheadrightarrow\Phi|_Y,T_Y)$ forms a geometric SEO. The inclusion $i_Y\colon Y\hookrightarrow X$ is one of its realizations.

Let $f\colon Y\to X$ be a bijection. Consider the data set $\Phi f$. For any $g$ in $M$, the function
$f^{-1}gf\colon Y\to Y$ is a  $\Phi f$-operation. Define $T\colon M\to \text{End}_{\Phi f}(Y)$ to map  $g$ in $M$ to $f^{-1}gf$. The incarnation $(\Phi f,T(M))$ is called the domain change of  $(\Phi,M)$ along $f$. The pair  $(-f\colon \Phi\to \Phi f, T)$ forms a  geometric SEO and 
 $f\colon Y\to X$
is one of its realizations.
\smallskip
 
 \noindent
 \textbf{Extending from a basis.}\quad SEOs can be effectively constructed using bases. 
\begin{proposition}\label{aDFADSGHFJDN}
  Let $(\Phi,  M)$ and $(\Psi, N)$ be incarnations and $\Omega$ be a basis of $(\Phi,  M)$.
  Then  two SEOs $(\alpha, T), (\alpha', T')\colon 
 (\Phi,  M)\to (\Psi, N)$  are equal if and only if $T=T'$ and $\alpha(\omega)=\alpha'(\omega)$ for any $\omega$ in $\Omega$.
\end{proposition}

\begin{proof}
 The only non trivial thing to prove in the statement of the proposition is that $\alpha=\alpha'$ when their restrictions to $\Omega$ are equal.
Assume $T=T'$ and $\alpha(\omega)=\alpha'(\omega)$ for any $\omega$ in $\Omega$.
 Since $\Omega$ generates $(\Phi,  M)$, any element in $\Phi$ is of the form $\phi=\omega g_1\cdots g_k$ for some 
  $\omega$ in $\Omega$ and a sequence of  elements
$g_1,\ldots, g_k$ in $M$. The assumption and the fact that $(\alpha, T)$ and $(\alpha', T)$ are SEOs, imply:
\[\alpha(\phi)=\alpha(\omega g_1\cdots g_k)=\alpha(\omega)T(g_1)\cdots T(g_k)=\]
\[=\alpha'(\omega)T(g_1)\cdots T(g_k)=\alpha'(\omega g_1\cdots g_k)=\alpha'(\phi).
\]
Consequently $\alpha=\alpha'$.
\end{proof}

According to  Proposition~\ref{aDFADSGHFJDN}, a SEO is determined by what it does on a basis of the domain. This is analogous to 
a linear map between  vector spaces being determined by  its values on a basis.   However unlike
for linear maps,  we cannot freely map elements of a basis of an incarnation to obtain a SEO. To obtain a SEO certain relations have to be preserved. Let  $(\Phi,  M)$ be an incarnation. A \textbf{relation} between measurements $\phi$ and $\psi$ in $\Phi$
 is by definition a pair of sequences 
 $\left( (g_1,\ldots,g_k),(h_1,\ldots,h_l)\right)$  of elements in $M$  for which the following equality holds:
 $\phi g_1\cdots g_k=\psi h_1\cdots h_l$.

\begin{proposition}\label{asfdfhdgfv}
Let $(\Phi,  M)$ and $(\Psi, N)$ be incarnations,  $\Omega$ be a basis of $(\Phi,  M)$, and 
 $\bar{\alpha}\colon\Omega\to \Psi$ and $T\colon M\to N$ be functions.
  \begin{enumerate}
 \item Assume that for every relation $\left( (g_1,\ldots,g_k),(h_1,\ldots,h_l)\right)$   between  any two elements $\omega$,  $\omega'$ in $\Omega$, the pair $\left( (T(g_1),\ldots,T(g_k)),(T(h_1),\ldots,T(h_l))\right)$ is a relation between $\alpha(\omega)$ and $\alpha(\omega')$  in $\Psi$. Under this assumption,    there is a unique SEO $(\alpha, T)\colon (\Phi,  M)\to (\Psi, N)$ for which the restriction of $\alpha\colon \Phi\to\Psi$ to $\Omega$ is ${\bar{\alpha}}$. 
 \item Assume $(\Phi,  M)$ and $(\Psi, N),$ are monoid incarnations, $T$ is a monoid homomorphism, and if $\omega g=\omega' h$ for some $\omega, \omega'$ in $\Omega$ and 
 $g,h$ in $M$, then  $\alpha(\omega) T(g)=\alpha(\omega')T(h)$. Under these assumptions,  there is a unique MEO $(\alpha, T)\colon (\Phi,  M)\to (\Psi, N)$ for which the restriction of $\alpha\colon \Phi\to\Psi$ to $\Omega$ is $\bar{\alpha}$. 
 \item Assume $(\Phi,  M)$ and $(\Psi, N)$ are group incarnations, $T$ is a group homomorphism, and if $\omega=\omega g$, for some
$ \omega$ in $\Omega$ and 
 $g$ in $M$, then  $\alpha(\omega) =\alpha(\omega)T(g)$.  Under these assumptions,   there is a unique GEO $(\alpha, T)\colon (\Phi,  M)\to (\Psi, N)$ for which the restriction of $\alpha\colon \Phi\to\Psi$ to $\Omega$ is $\bar{\alpha}$. 
 \end{enumerate}
 \end{proposition}
\begin{proof}
Since the proofs are analogous, we illustrate only how to show statement (2).
For every $\phi$  in $ \Phi$, there exist (not necessarily unique) $\omega$ in $ \Omega$ and $g$ in $ M$ such that $\phi = \omega g$. 
The assumption implies that  the expression $\alpha(\omega) T(g)$ depends on $\phi$ and not on the choices of $\omega$ and $ g$
for which  $\phi = \omega g$.   Thus by mapping $\phi$ in $\Phi$ to $\alpha(\omega) T(g)$ in $\Psi$, we obtain a well defined  function also denoted by
 $\alpha \colon \Phi \to \Psi$.  The pair $(\alpha,T)$ is the desired MEO.
 The uniqueness is a consequence of  Proposition~\ref{aDFADSGHFJDN}.
\end{proof}

 For example assume $(\Phi,  M)$ is  a  transitive group incarnation and $(\Psi,  N)$ is a  group incarnation. Choose an element $\omega$ in $\Phi$. Recall that any such element  is a basis of  $(\Phi,  M)$.  Fix a group homomorphism $T\colon M\to N$.
 Then  any GEO   $(\alpha,T)\colon (\Phi,  M)\to (\Psi,  N)$ is uniquely determined by the element $\alpha(\omega)$ in $\Psi$.
 Thus by choosing a basis element $\omega$ in $\Phi$, we can  identify  the collection of 
 GEOs of the form $(\alpha,T)\colon (\Phi,  M)\to (\Psi,  N)$
  with a subset of $\Psi$. To describe this  subset explicitly, we apply  Proposition~\ref{asfdfhdgfv}.2.  It states that there is a GEO  $(\alpha,T)\colon (\Phi,  M)\to (\Psi,  N)$ (necessarily unique) such that $\alpha(\omega)=\psi$ if and only if the following implication holds:
if $\omega=\omega g$, then   $\psi=\psi T(g)$. The collection $M_\omega:=\{g\in M\ |\ \omega=\omega g\}$ is  the isotropy subgroup of $\omega$ 
consisting of all the elements  in $M$ that fix $\omega$. Thus GEOs of the form  $(\alpha,T)\colon (\Phi,  M)\to (\Psi,  N)$ can be identified  with the subset of all the elements in $\Psi$ whose isotropy group contains  $T(M_\omega)$.
\section{Decomposition}\label{dsvdfhgfn}
Let  $(\Phi,M)$ be an incarnation of a data set $\Phi$.  Consider its quotient $\Phi/M$, which is a partition of  $\Phi$, and the block incarnations
$(\Psi,M)$ for every block $\Psi$ in $\Phi/M$ (see Section~\ref{asfadfhg}). Let $X$ be the domain of $\Phi$. Recall that the domain of the data set  $\coprod_{\Psi\in \Phi/M} \Psi$ is given by the disjoint union $\coprod_{\Psi\in \Phi/M} X$, and  that this data set consists of  functions 
$\coprod_{\Psi\in \Phi/M} X\to \mathbf{R}$  whose restrictions to all but one  summands $X$ in $\coprod_{\Psi\in \Phi/M} X$ 
is the $0$ function and the restriction to the remaining summand belongs to the   corresponding block of the partition $\Phi/M$. Define:
\[M'=\left\{\coprod_{\Psi\in \Phi/M} g\colon \coprod_{\Psi\in \Phi/M} X\to \coprod_{\Psi\in \Phi/M} X\ \ \ \  |\ \ \ \ \   g\in M\right\}.\]
Then $M'\subset \text{End}_{\coprod_{\Psi\in \Phi/M} \Psi}(\coprod_{\Psi\in \Phi/M} X)$. We call
$(\coprod_{\Psi\in \Phi/M} \Psi, M')$ the diagonal incarnation.  Define $T\colon M\to M'$ to map $g\colon X\to X$ in  $M$ to
$\coprod_{\Psi\in \Phi/M} g$ in $M'$. Define $\alpha\colon \Phi\to \coprod_{\Psi\in \Phi/M} \Psi$ to map $\phi$ to the function 
$\coprod_{\Psi\in \Phi/M} X\to \mathbf{R}$ whose restriction to the summand $X$ corresponding to the block $[\phi]$ is $\phi$ and that maps all other summands to $0$. 
Note that both of the functions $\alpha$ and $T$ are bijections. Furthermore they form a SEO between $(\Phi,M)$ and $(\coprod_{\Psi\in \Phi/M} \Psi, M')$.

 \begin{proposition}\label{adgsdghf}
 The SEO $(\alpha,T)\colon (\Phi,M)\to (\coprod_{\Psi\in \Phi/M} \Psi, M')$ is an isomorphism.
  \end{proposition}

\section{Grothendieck graphs}
 In this section we explain a convenient  data structure   to encode  incarnations of  data sets.

A \textbf{Grothendieck graph} is a triple  $(V,M,E)$ consisting of a finite set $V$
  whose elements are called vertices, a finite set $M$
 whose elements are called colors or operations,  and a  subset
 $E\subset V\times M\times V$ whose elements are called edges, such that, for every vertex $v$ in $V$, the following composition is a bijection:
 \[\begin{tikzcd}
 (\{v\}\times M\times V)\cap E\ar[hook]{r} & E\ar[hook]{r} & V\times M\times V\ar{r}{\text{pr}_M} & M.
 \end{tikzcd}\]
This condition assures that, for every $v$ in $V$ and $g$ in $M$, there is a unique element  in $V$, denoted by $vg$,
 such that $(v,g,vg)$ is an edge in  $E$. 
 For example let $(\Phi,M)$ be an incarnation of a data set $\Phi$. Define:
  \[E_{\Phi,M}:=\{(\phi,g,\psi)\in \Phi\times M\times\Phi\ |\ \phi g=\psi\}.\]
 Then the triple $(\Phi,M,E_{\Phi,M})$ is a Grothendieck graph. We think about this graph as  a convenient data structure representing   the incarnation $(\Phi,M)$.
  
  Grothendieck graphs are also convenient  to represent SEOs.
 Define a  \textbf{morphism between  Grothendieck  graphs} $(V,M,E)$ and  $(W,N,F)$ to be a pair of functions
 $\alpha\colon V\to W$ and $T\colon M\to N$ such that, if $(v,g,w)$ belongs to $E$, then $(\alpha(v),T(g),\alpha(w))$ belongs to $F$.
 Such a morphism is denoted as $(\alpha,T)\colon (V,M,E)\to (W,N,F)$.
 Componentwise composition 
 defines  a category structure on the collection  of Grothendieck  graphs and we use the symbol GGraph  to denote this category.
 If $(\alpha,T)\colon (\Phi,M)\to (\Psi,N)$ is a SEO, then $(\alpha, T)\colon (\Phi,M,E_{\Phi,M})\to (\Psi,N,E_{\Psi,N})$
 is a morphism between the associated Grothendieck graphs. By  assigning to  a SEO $(\alpha,T)$  the 
graph morphism given by the same pair $(\alpha,T)$, we obtain a  fully faithful functor   from the category  $\text{Nirvana}$ to $\text{GGraph}$. 

  
  Grothendieck graphs can also be used to encode  pseudometric information on incarnations. A pseudometric on  a Grothendieck graph  $(V,M,E)$ is a pseudometric $d$ on $V$ such that  $d(v,w)\geq d(vg,wg)$ for all
 $v$ and  $w$ in $V$, and $g$ in $M$. 
 For example, the pseudometric
 $\|\phi-\psi\|_\infty$ on $\Phi$ is a pseudometric on the graph $(\Phi,M,E_{\Phi,M})$.

A Grothendieck  graph $(V,M,E)$ is said to be compatible with a  monoid structure on $M$ if $(v,1,v)$ is  in $E$, and  whenever $(v_0,g_0,v_1)$ and $(v_1,g_1,v_2)$
 belong to $E$, then so does $(v_0,g_1g_0,v_2)$. In this case the composition operation given by the association   $(v_0,g_0,v_1) (v_1,g_1,v_2)\mapsto (v_0,g_1g_0,v_2)$ defines a category  structure, denoted by $\text{Gr}_MV$,  with  $V$ as the set of objects and   $E$ as the set of morphisms. This category is a  familiar Grothendieck construction~\cite{MR1926776, MR510404}. For example,
  the  Grothendieck graph associated  with  a  monoid incarnation $(\Phi,M)$  is compatible with the  monoid structure on  $M$.
  We think about   $\text{Gr}_{M}\Phi$  as an additional structure on the data set $\Phi$:
objects are the measurements  in $\Phi$, morphisms are  triples  $(\phi,g,\phi g)$, where $\phi$ is  in $\Phi$,   $g$ is  in $M$, and the composition of $ (\phi,g,\phi g)$ and $(\phi g,h,\phi gh)$  is given by
$(\phi,gh,\phi gh)$.


  A \textbf{contravariant functor} indexed by a Grothendieck  graph $(V,M,E)$ with values in a category $\mathcal{C}$,
 denoted  as $P\colon (V,M,E)\to \mathcal{C}$, is by definition a sequence of objects $\{P(v)\ |\ v\in V\}$  and  a sequence of morphisms
 $\{P(v_0,g,v_1)\colon P(v_1)\to P(v_0)\ |\ (v_0,g,v_1)\in E\}$ in $\mathcal{C}$ subject to: if $(v_0,g_0,v_1)$, $(v_1,g_1,v_2)$, and $(v_0,h,v_2)$ are edges
 in $E$, then $P(v_2,h,v_0)=P(v_2,g_1,v_1)P(v_1,g_0,v_0)$. 
 If  $(V,M,E)$ is  compatible with a  monoid structure on $M$, then a  contravariant functor indexed  by $(V,M,E)$
 is simply a contravariant functor indexed by the category $\text{Gr}_MV$. 
 
 Let $(\Phi,M)$ be an  incarnation of a data set $\Phi$ consisting of measurements on $X$, and  $(\Phi,M,E_{\Phi,M})$  be the associated Grothendieck graph. 
 For every $g$ in $M$, the function $-g\colon\Phi\to\Phi$, mapping $\phi$ to $\phi g$, is geometric and realized by $g\colon X\to X$
 (see Section~\ref{aSFDFGH}).  Persistent homology  leads therefore to the following collections of objects and morphisms in $\text{Tame}( [0,\infty)\times  {\mathbf R},\text{Vect})$ as explained in Section~\ref{aSFDFGH}: 
 \[\left\{PH_d^{\Phi}(\phi)\  |\  \phi\in\Phi\right\},\]
 \[\left\{PH_d^{-g}(\phi)\colon PH_d^{\Phi}(\phi g)\to  PH_d^{\Phi}(\phi)\  |\  (\phi,g,\phi g)\in E_{\Phi,M}\right\}.\]
 These sequences form a  functor $PH^{\Phi}_d\colon (\Phi,M,E_{\Phi,M})\to \text{Tame}( [0,\infty)\times  {\mathbf R},\text{Vect})$
 also referred to as the persistent homology functor of the incarnation $(\Phi,M)$.
 
  Let  $(\alpha,T)\colon  (W,N,F)\to (V,M,E)$ be a morphism and   $P\colon (V,M,E)\to \mathcal{C}$ be a functor.
 The following sequences of objects and morphisms in $\mathcal{C}$ form  a  contravariant functor  denoted by $P(\alpha,T)\colon (W,N,F)\to \mathcal{C}$ and called the \textbf{composition} of  $(\alpha,T)$ with $P$:
  \[\left\{P(\alpha(v))\ |\ v\in V\right\},\]
  \[ \left\{P(w_0,g,w_1)\colon P(\alpha(w_1))\to P(\alpha(w_0))\ |\ (w_0,g,w_1)\in F\right\}.\]
   For example, let  $(\text{id}_{\Phi},i)\colon (\Phi,M)\to (\Phi,\text{End}_{\Phi}(X))$ be the canonical SEO (see Section~\ref{dfadfhbsfg}). Consider  the induced morphism
  of the associated Grothendieck graphs:
  \[(\text{id}_{\Phi},i_M)\colon (\Phi,M,E_{\Phi,M})\to (\Phi, \text{End}_{\Phi}(X),E_{\Phi,\text{End}_{\Phi}(X)}).\]
  Consider also  the persistent homology of the universal incarnation:
  \[PH^{\Phi}_d\colon (\Phi,\text{End}_{\Phi}(X),E_{\Phi,\text{End}_{\Phi}(X)})\to \text{Tame}( [0,\infty)\times  {\mathbf R},\text{Vect}).\]
  The composition of these two functors coincides with the
  persistent homology  of the incarnation $(\Phi,M)$:
 \[PH^{\Phi}_d\colon (\Phi,M,E_{\Phi,M})\to \text{Tame}( [0,\infty)\times  {\mathbf R},\text{Vect}).\] 
 In this way we obtain a commutative diagram:
 \[ \begin{tikzcd}[column sep=25]
  &  (\Phi, \text{End}_{\Phi}(X),E_{\Phi,\text{End}_{\Phi}(X)})\ar[bend left=8]{rd}{PH^{\Phi}_d}  & [-30pt]\\
  (\Phi,M,E_{\Phi,M})\ar[bend left=8]{ru}{(\text{id}_{\Phi},i_M)}\ar{rr}{PH^{\Phi}_d} & &  \text{Tame}( [0,\infty)\times  {\mathbf R},\text{Vect})
 \end{tikzcd}\]
   
  Such a commutativity does not hold for arbitrary SEOs.  Consider a SEO  $(\alpha,T)\colon (\Phi,M)\to (\Psi,N)$. 
  We can form two functors indexed by the graph  $(\Phi,M,E_{\Phi,M})$:
  \[\begin{tikzcd}[row sep=4, column sep=24]
 &  &\text{Tame}( [0,\infty)\times  {\mathbf R},\text{Vect})\\
  (\Phi,M,E_{\Phi,M})\ar[bend left=8, end anchor=west]{urr}{PH_d^\Phi} \ar[bend right=15,end anchor=west]{rd}{\alpha}\\
  & (\Psi,N,E_{\Psi,N})\ar{r}{PH_d^\Psi} & \text{Tame}( [0,\infty)\times  {\mathbf R},\text{Vect})
 \end{tikzcd}\]
These functors  rarely  coincide. 
 However, in the case  $(\alpha, T)$ is geometric,   the morphisms $PH_d^{\alpha}(\phi)\colon PH_d^{\Psi}(\alpha(\phi))\to  PH_d^{\Phi}(\phi)$
 (see Section~\ref{aSFDFGH}),
 for all $\phi$ in $\Phi$, 
   form a natural transformation.

\section{Conclusions}
In the following figure we give a graphical representation  of some of the concepts introduced in this article.
Data sets can be equipped with three structures: a pseudometric, an incarnation describing an action, and a Grothendieck graph. 
We imagine Nirvana as the landscape of all possible incarnations of data sets, represented  by the shaded region in the following figure.
Each point  in Nirvana  has  a lot of internal structure allowing the extraction of persistent homology.
In this landscape the black arrows represent   geometric SEOs and the grey ones non-geometric SEOs. 
Recall that  geometric SEOs enable us to compare relevant persistent homology.  Non-geometric SEOs contain complementary  information.
\begin{figure}[h!]
	\centering
	\includegraphics[scale=0.53]{picture.pdf}	
\label{embedings}
\end{figure}


\section*{Acknowledgements}
The research carried out by N.Q. was partially supported by GNSAGA-INdAM (Italy).
A.D. has been supported by the SmartData@PoliTO center on Big Data and Data Science and by the Italian MIUR Award ``Dipartimento di Eccellenza 2018-202'' - CUP: E11G18000350001. The work of W.C. and F.T. was partially supported by the Wallenberg AI, Autonomous System and Software Program (WASP) funded by Knut and Alice Wallenberg Foundation. The work of W.C. was also in part funded by VR and G\"oran Gustafsson foundation.

\newpage 

\bibliographystyle{plainurl}
\bibliography{references_copy}
\end{document}